\documentclass[a4paper,10pt]{article}
\usepackage[utf8]{inputenc}
\usepackage{amsfonts}
\usepackage{amsthm}
\usepackage{amsmath}
\usepackage{setspace}
\usepackage{array}
\usepackage{setspace}
\usepackage[margin=3cm]{geometry}
\numberwithin{equation}{section}
\usepackage{longtable}
\newtheorem{theorem}{Theorem}[section]
\newtheorem{lemma}[theorem]{Lemma}

\theoremstyle{definition}

\theoremstyle{remark}
\newtheorem{remark}[theorem]{Remark}

\numberwithin{equation}{section}

\makeatletter
\providecommand{\tabularnewline}{\\}
\makeatother

\title{The Prescribed Ricci Curvature Problem on Three-Dimensional Unimodular Lie Groups}
\author{Timothy Buttsworth}
\begin{document}
\maketitle
\begin{abstract}
 Let $G$ be a three-dimensional unimodular Lie group, and let $T$ be a left-invariant symmetric 
 $(0,2)$-tensor field on $G$. We provide the necessary and sufficient conditions on $T$ for the existence of a pair 
 $(g,c)$ consisting of a left-invariant Riemannian metric $g$ and a positive constant $c$ such that $Ric(g)=cT$, where 
 $Ric(g)$ is the Ricci curvature of $g$. We also discuss the uniqueness of such pairs and show that, in most cases, there exists at most one 
 positive constant $c$ such that $Ric(g)=cT$ is solvable for some left-invariant Riemannian metric $g$.
\end{abstract}

\maketitle
\section{Introduction}
A problem of fundamental interest in Riemannian geometry is finding a 
Riemannian metric $g$ to satisfy the prescribed Ricci curvature equation
\begin{equation}\label{PRCE}
 Ric(g)=T
\end{equation}
for some fixed symmetric $(0,2)$-tensor field $T$ on a manifold $M$. In 
\cite{DeT}, DeTurck shows that if $T$ is non-degenerate at a point $p\in M$, then there is a Riemannian metric 
$g$ solving equation \eqref{PRCE} in some neighbourhood of $p$. DeTurck and Goldschmidt in \cite{Gold} also show
that \eqref{PRCE} holds in a neighbourhood of a point if 
$T$ has constant rank and satisfies certain other constraints. Further results on local existence are available in \cite{Besse,APB,APC}. 

One would like to know when it is possible to find a Riemannian metric $g$ such that equation 
\eqref{PRCE} holds on all of $M$, and not merely on some neighbourhood. Many results on this problem rely on the inverse function theorem, for example, see 
 \cite{DeT2,Delay01,Delay02,Godfrey}. Older 
results on the global solvability of \eqref{PRCE} are available in~\cite{Besse}, and more recent progress can be found in, for example, 
\cite{Pina,Pina2,APC,HMPRC}, 
but in general, finding results on the 
global solvability of \eqref{PRCE} is appears to be challenging. 

The work of Milnor in \cite{Milnor} contains results on the possible signatures the Ricci curvature of left-invariant Riemannian metrics can have 
on Lie groups. Subsequent to Milnor's results, a lot of work has been done investigating
the Ricci curvature of left-invariant metrics. For example, in \cite{SCLG}, Ha and Lee complete Milnor's classification of the available signatures of the Ricci 
curvature of left-invariant metrics on three-dimensional Lie groups, and in
\cite{Krem}, Kremlev and Nikonorov also investigate the available signatures of the Ricci curvature, but on four dimensional Lie groups. 
In \cite{RicEig}, Kowalski and Nikcevic find the possible eigenvalues of the Ricci curvature of left-invariant metrics on
three-dimensional Lie groups. However, the Ricci curvature cannot be recovered from merely the signature or the eigenvalues 
so these results do not provide solutions of \eqref{PRCE}. 

One result provided by 
DeTurck and Hamilton in \cite{DS2} and \cite{Hamilton2}, respectively, shows that if $T$ is a positive-definite $(0,2)$-tensor field on the sphere $\mathbb{S}^2$, then 
\eqref{PRCE} is solvable if and only if the volume of $T$ is $4\pi$. This result indicates that, instead 
of finding a Riemannian metric $g$ to solve \eqref{PRCE}, it may be reasonable to look for a Riemannian metric $g$ and a positive constant $c$ such that 
\begin{equation}\label{PRCEC}
 Ric(g)=cT.
\end{equation}
Indeed, in \cite{Hamilton2}, Hamilton shows that if $T$ is left-invariant and positive-definite on the Lie group $SO(3)$, then there is a unique 
$c>0$ and a left-invariant Riemannian metric $g$, unique up to scaling, such that \eqref{PRCEC} holds. 
Similarly, Pulemotov shows in \cite{HMPRC} that if $H$ is a maximal connected Lie subgroup of a compact Lie group $G$, 
then there is a solution of \eqref{PRCEC} for any invariant positive semi-definite $T$ on the 
homogeneous space 
$G/H$. Furthermore, as shown in \cite{HMPRC,Rubin}, the pair $(g,c)$ is unique up to the scaling of $g$ when there 
are two inequivalent irreducible summands in the isotropy representation.

The results in \cite{DeTKoi,DS2,Hamilton2,HMPRC,Rubin} only
apply to closed manifolds; there has been very little work on the global solvability of 
\eqref{PRCEC} on noncompact manifolds. In this paper, we resolve the question of global solvability of the prescribed Ricci curvature on three-dimensional unimodular Lie groups $G$ where 
$T$ and $g$ are required to be left-invariant, thus extending the work of Hamilton in \cite{Hamilton2} to certain noncompact manifolds and arbritrary signatures of $T$.
We also remark that our results imply the results about the available signatures of the Ricci cuvature of left-invariant metrics on three 
dimensional Lie groups in \cite{Milnor} and \cite{SCLG}, but do not imply, nor are implied by the results of \cite{RicEig}.

\section{Summary of Results}
Let $G$ be a three-dimensional and unimodular Lie group, and let $\mathfrak{g}$ be the Lie algebra of $G$. Milnor shows in \cite{Milnor} that any 
left-invariant Riemannian metric
$g$ on $G$ is diagonalisable in a basis $\{V_1,V_2,V_3\}$ of $\mathfrak{g}$ such that, if 
$\epsilon_{ijk}$ is the Levi-Civita symbol, the Lie bracket relations are 
\begin{equation}
 [V_i,V_j]=\sum_{k=1}^{3}\epsilon_{ijk}\lambda_kV_k\label{NICE}
\end{equation}
for some numbers $\lambda_k$. 
By scaling and reordering the basis elements, we can assume that $\lambda_k\in \{-2,0,2\}$, 
that there are more non-negative values of $\lambda_k$ than 
negative values, that $\lambda_3=0$ if any 
zeroes appear and that $\lambda_1=2$, unless $\lambda_1=\lambda_2=\lambda_3=0$. 
The possibilities for $\lambda_1,\lambda_2,\lambda_3$ with these extra 
constraints are then in one-to-one correspondence with the different three-dimensional unimodular Lie groups. 

Our results are stated in Theorems \ref{SO3T}, \ref{SL2T}, \ref{E2T}, \ref{E11T} and \ref{HT} below. 
Together these theorems tell us the following. 
\begin{theorem}
 Let $T$ be a given left-invariant $(0,2)$ tensor field. There exists a pair 
$(g,c)$, where $g$ is a left-invariant Riemannian metric and 
$c>0$, such that $Ric(g)=cT$, if and only if both of the following statements hold:

(i) $T$ is diagonalisable in a basis $V_1,V_2,V_3$ which satisfies 
\eqref{NICE}
 
 (ii) the diagonal components $T_1,T_2,T_3$ satisfy one of the constraints in Table \ref{table} depending on the Lie group $G$ we study.
\end{theorem}
 For each possibility, the uniqueness of $g$ and $c$ in the pair is studied. Note that if $(g,c)$ is a solution, 
 then $(c'g,c)$ is also a solution for any $c'>0$ because of the scaling invariance of the Ricci curvature. Therefore, when we talk about 
the uniqueness of $(g,c)$, we always mean uniqueness up to scaling of $g$. For example, when we write $g_1\sim g_2$ in Table \ref{table}, we mean that 
 $g_1$ is a scalar multiple of $g_2$. Similarly, if we say our solution $(g,c)$ is unique, then we mean that if we find another solution $(\bar{g},\bar{c})$, then 
$\bar{c}=c$ and $\bar{g}$ is a scalar multiple of $g$. 

Let us make some brief remarks about some subtleties associated with the table. 
\begin{remark}
It is important to note that we are not treating uniqueness up to isometries of Lie algebras, only up to scaling. For instance, it is well understood that 
all left-invariant metrics on $\mathbb{R}^3$ are isometric, so constitute the same geometry, but for our purposes, we will treat these metrics as distinct. This is why we say that the solutions of $Ric(g)=0$ are non-unique on $\mathbb{R}^3$, even though all of the solutions are isometric. 
\end{remark}
\begin{table}
\caption{}
\label{table}
\centering
\begin{tabular}{ |c|c|c|c|c| }
 \hline
Lie Group  & Signature of & Necessary and 
& Is $c$ the & 
$g_1 \sim g_2$ for any \tabularnewline
$\left(\lambda_{1},\lambda_{2},\lambda_{3}\right)$ & $(T_1,T_2,T_3)$ & sufficient  conditions 
&same for all & solutions $(g_1,c)$\tabularnewline
& & on $\left(T_{1},T_{2},T_{3}\right)$ & solutions? &  and $(g_2,c)$ \tabularnewline
\hline 
$SO\left(3\right)$ & $\left(+,+,+\right)$ & - & Yes & Yes\tabularnewline
$\left(2,2,2\right)$ & $\left(+,0,0\right)$ & - & Yes & No\tabularnewline
 & $\left(+,-,-\right)$ & See Theorem \ref{SO3T} & No & Yes\tabularnewline
\hline 
$SL\left(2\right)$ & $\left(+,-,-\right)$ & $T_{3}+T_{1}>0$ & Yes & Yes\tabularnewline
$\left(2,2,-2\right)$ & $\left(-,+,-\right)$ & $T_{3}+T_{2}>0$ & Yes & Yes\tabularnewline
 & $\left(-,-,+\right)$ & $\max\left\{ -T_{1},-T_{2}\right\} <T_{3}$ & Yes & Yes\tabularnewline
 &  & $\min\left\{ -T_{1},-T_{2}\right\} >T_{3}$ & Yes & Yes\tabularnewline
 &  & $T_{3}=-T_{1}=-T_{2}$ & Yes & No\tabularnewline
 & $\left(-,0,0\right)$ & - & Yes & No\tabularnewline
 & $\left(0,-,0\right)$ & - & Yes & No\tabularnewline
\hline 
$E\left(2\right)$ & $\left(0,0,0\right)$ & - & No & No\tabularnewline
$\left(2,2,0\right)$ & $\left(+,-,-\right)$ & $T_{1}+T_{2}>0$ & Yes & Yes\tabularnewline
 & $\left(-,+,-\right)$ & $T_{1}+T_{2}>0$ & Yes & Yes\tabularnewline
\hline 
$E\left(1,1\right)$ & $\left(0,0,-\right)$ & - & Yes & No\tabularnewline
$\left(2,-2,0\right)$ & $\left(+,-,-\right)$ & $T_{1}+T_{2}>0$ & Yes & Yes\tabularnewline
 & $\left(-,+,-\right)$ & $T_{1}+T_{2}>0$ & Yes & Yes\tabularnewline
\hline 
$H_3$ & $\left(+,-,-\right)$ & - & Yes & Yes\tabularnewline
$(2,0,0)$ & & & & \tabularnewline
\hline 
$\mathbb{R}^{3}$ & $\left(0,0,0\right)$ & - & No & No\tabularnewline
$(0,0,0)$ & & & & \tabularnewline
\hline
\end{tabular}
\end{table}

\begin{remark}
In the table, the signature of $(T_1,T_2,T_3)$ is not simply the signature 
of the tensor $T$, but the specific signs of $T_1,T_2$ and $T_3$ when diagonalised in a basis with $\lambda_k$ values given in the exact order appearing in the table. This is why, for instance, the signatures $(+,-,-)$ and $(-,+,-)$ both appear for $E(2)$. 
For $SO(3)$ however, we have the ordering convention $T_1\ge T_2\ge T_3$ (see section 4 for more details), so this subtlety does not come up. For example, 
on $SO(3)$, the results concerning the signature $(+,0,0)$ also imply the results for the signatures $(0,+,0)$ and $(0,0,+)$. 
\end{remark}
One result that can be seen from Table \ref{table} is the following theorem. 
\begin{theorem}\label{CU}
 Let $G$ be a $3$-dimensional unimodular Lie group and let $T$ be a left-invariant symmetric $(0,2)$ tensor. Unless 
 $T=0$ or $T$ has signature $(+,-,-)$ and $G=SO(3)$, there is at most one positive constant $c>0$ such that 
 there exists a left-invariant Riemannian metric $g$ solving $Ric(g)=cT$. 
\end{theorem}
\section{The Ricci Curvature of Invariant Metrics}
To prove the results 
in Table \ref{table} we must first compute the Ricci curvature of an arbitrary 
 left-invariant metric $g$ on the three-dimensional unimodular Lie group $G$. See \cite{Hamilton2} for closely related 
 computations. 
 For a given left-invariant metric $g$, 
 choose a basis $V_1,V_2,V_3$ of $\mathfrak{g}$ such that 
\eqref{NICE} holds in which $g$ is diagonal. Therefore, 
$g(V_i,V_j)=\delta_j^i v_i$ for some positive numbers $v_1,v_2,v_3$, where 
$\delta^i_j$ is the Kronecker symbol. 
We then compute the Ricci curvature and find 
\begin{align*}
 Ric(V_1,V_1)=\frac{2x_2x_3}{v_2v_3},\qquad
 Ric(V_2,V_2)=\frac{2x_1x_3}{v_1v_3},\qquad
 Ric(V_3,V_3)=\frac{2x_1x_2}{v_1v_2}
\end{align*}
where 
\begin{align*}
x_1=\frac{\lambda_2 v_2+\lambda_3 v_3-\lambda_1 v_1}{2},
x_2&=\frac{\lambda_1 v_1+\lambda_3 v_3-\lambda_2 v_2}{2},
x_3=\frac{\lambda_1 v_1+\lambda_2 v_2-\lambda_3 v_3}{2};
\end{align*}
and $R(V_i,V_j)=0$ whenever $i\neq j$, cf. \cite[section 6]{Hamilton2}. Therefore, 
the Ricci curvature of $g$ in a basis where $g$ is diagonal is also diagonal.
 
This tells us immediately that for $Ric(g)=cT$ to hold, since $g$ is diagonalisable, $g$ and $T$ must be simultaneously 
diagonalisable in a basis $V_1,V_2,V_3$ which satisfies~\eqref{NICE}. 
In the search for a left-invariant metric $g$ and positive constant $c$ to solve the equation $Ric(g)=cT$, it therefore suffices 
to look for metrics on $\mathfrak{g}$ which 
are diagonal in some basis $V_1,V_2,V_3$ satisfying \eqref{NICE} in which $T$ is also diagonal, say 
\begin{equation}\label{TV}
 T(V_i,V_j)=\delta_j^iT_i.
\end{equation}
Once such a basis 
$V_1,V_2,V_3$ is found, solving $Ric(g)=cT$ for $g$ diagonal in the basis $V_1,V_2,V_3$ is equivalent to solving 
the system of equations 
\begin{equation}
 \frac{2x_jx_k}{v_jv_k}=cT_i \label{THEQc}
\end{equation}
for $i=1,2,3$ where for a given $i$, the indices $j,k\in\{1,2,3\}$ are chosen such that $i,j,k$ are pairwise distinct. Also note that $v_i$
are required to be positive. Due to the scaling invariance of the Ricci curvature, if we have a solution $(v_1,v_2,v_3,c)$ of 
this system of equations, then we can scale $(v_1,v_2,v_3)$ so that the additional constraint $v_1v_2v_3c=1$ is satisfied. 
If this constraint is satisfied, then \eqref{THEQc} becomes
\begin{equation}
 2v_ix_jx_k=T_i, \label{THEQ}
\end{equation}
so we can study this system of equations as an alternative of \eqref{THEQc}. Note that if we find a solution 
$v_1,v_2,v_3$ of \eqref{THEQ}, 
we can recover the associated $c$ from the formula $v_1v_2v_3c=1$, finding a solution of \eqref{THEQc}; cf. \cite[section 6]{Hamilton2}. 

This reasoning is summarised in the following.
\begin{lemma}\label{BS}
Let $T$ be a left-invariant $(0,2)$ tensor field on a three-dimensional 
unimodular Lie group $G$. There exists a left-invariant Riemannian metric 
$g$ and a constant $c>0$ solving \eqref{PRCEC} if and only if the following conditions hold: 

(i) there exists a basis $V_1,V_2,V_3$ of the Lie algebra $\mathfrak{g}$ satisfying \eqref{NICE} in which $T$ is diagonal 

(ii) there is a solution $(v_1,v_2,v_3,c)$ of \eqref{THEQc}, or equivalently, a solution 
$(v_1,v_2,v_3)$ of \eqref{THEQ} with $T_i$ given by \eqref{TV}.
\end{lemma}

In the following sections, we will analyse the solvability of \eqref{THEQc}, or equivalently,~\eqref{THEQ}, on the different three-dimensional unimodular Lie groups 
for 
$T_i$ given in \eqref{TV}. 
Note that when discussing solutions of \eqref{THEQc}, we will use the same uniqueness convention as we do for solutions of $Ric(g)=cT$. Namely, we will only discuss uniqueness of solutions 
of \eqref{THEQc} up to scaling of the triplet $(v_1,v_2,v_3)$. 

Once we have found solutions of \eqref{PRCEC}, we will verify the conditions of the following lemma to study 
the uniqueness of solutions. 
\begin{lemma}\label{UNIQUE}
Let $T$ be a non-zero left-invariant $(0,2)$ tensor field on $G$ and let $g$ be a left-invariant metric such that $Ric(g)=cT$ for some $c>0$. 
Then $(g,c)$ is the unique pair such that 
$Ric(g)=cT$ if both of the following conditions hold:

(i) $g$ is diagonal in any basis $X_1,X_2,X_3$ in which 
$T$ is diagonal and \eqref{NICE} is satisfied 

(ii) the solution 
$v_1,v_2,v_3$ of \eqref{THEQ} is unique whenever 
$T_1,T_2,T_3$ are the components of $T$ in a $X_1,X_2,X_3$ basis satisfying \eqref{NICE}.

Furthermore, (ii) is a necessary condition for $(g,c)$ to be the unique pair solving 
$Ric(g)=cT$. 
\end{lemma}
\begin{proof}
Suppose that these two conditions were true and suppose that we had another left-invariant metric $\overline{g}$ such that $Ric(\overline{g})=\overline{c}T$ 
 for some positive constant $\overline{c}$. Let $X_1,X_2,X_3$ be a basis in which 
 $\overline{g}$ is diagonal, $T$ is diagonal and \eqref{NICE} holds. We know that 
 $g$ is also diagonal in this basis by condition~$(i)$. 
 
 If we let $v_1,v_2,v_3$ be the components of $g$ and $\overline{v_1},\overline{v_2},\overline{v_3}$ be the components of $\overline{g}$ in the 
 $X_1,X_2,X_3$ basis, then 
 $(v_1,v_2,v_3,c)$ and $(\overline{v_1},\overline{v_2},\overline{v_3},\overline{c})$ 
 are both solutions of \eqref{THEQc} with $T_i=T(X_i,X_i)$. 
 After possibly scaling both $g$ and $\overline{g}$ we can assume that 
 $v_1v_2v_3c=1$ and $\overline{v_1}\overline{v_2}\overline{v_3}\overline{c}=1$. This ensures  that 
 $v_1,v_2,v_3$ and $\overline{v_1},\overline{v_2},\overline{v_3}$ are both solutions of 
 \eqref{THEQ}. 
 The solution of \eqref{THEQ} is unique by condition $(ii)$, which means that 
 $g$ is the same as $\overline{g}$ up to scaling. We also see that $cT=Ric(g)=Ric(\overline{g})=\overline{c}T$ so $c=\overline{c}$ as well. 
 
It is clear that if $(ii)$ does not hold, we get multiple solutions $(g,c)$ of $Ric(g)=cT$. 
\end{proof}
\begin{remark}\label{+--an}
If $G=SO(3)$ and $T$ has signature $(+,-,-)$, we find from Lemma~\ref{SO3+--} that \eqref{THEQ} has at most two solutions. 
Lemma \ref{UNIQUE} can be generalised in the obvious way to show that, in this case,  
there are at most two solutions of $Ric(g)=cT$.
\end{remark}

\section{The Special Orthogonal Group $SO(3)$}
We have $\lambda_k=2$ for $k=1,2,3$. Theorem \ref{SO3T} below is our main result of this 
section. Note that if 
$T$ is diagonal in a basis satisfying~\eqref{NICE}, it is also diagonal in a basis 
$V_1,V_2,V_3$ satisfying~\eqref{NICE} such that
$T(V_1,V_1)\ge T(V_2,V_2)\ge T(V_3,V_3)$ because $\lambda_1=\lambda_2=\lambda_3$. This follows by reordering and rescaling the basis vectors. Therefore, for Theorem \ref{SO3T} and its proof, 
we assume without loss of generality that the diagonal entries of $T$ are ordered in this way.

\begin{theorem}\label{SO3T}
Let $T$ be a left-invariant $(0,2)$ tensor field on $SO(3)$. There 
exists a left-invariant Riemannian metric $g$ and a positive constant $c>0$ such that 
$Ric(g)=cT$ if and only if $T$ is diagonalisable with  $T=diag(T_1,T_2,T_3)$ in a basis $V_1,V_2,V_3$ satisfying \eqref{NICE} with 
$\lambda_1=\lambda_2=\lambda_3=2$ and  one of the following conditions is satisfied: 

(i) $T_1\ge T_2\ge T_3>0$

(ii) $T_1>0$, $T_2=T_3=0$ 

(iii) $T_1>0$, $T_3\le T_2<0$ and either $\bar{f}(-\frac{T_1+T_2+T_3}{3})\ge0$ while $-\frac{T_1+T_2+T_3}{3}<0$, or else $\bar{f}(-T_1)>0$ where 
 \begin{align*}
  \bar{f}(p)=2p^3+(T_1+T_2+T_3)p^2-T_1T_2T_3.
 \end{align*}

In case (i), the solution $(g,c)$ is unique. 
In case (ii), there are infinitely many solutions $(g,c)$, but $c$ is the same for all solutions. In case (iii), the solution $(g,c)$ is unique unless 
 $\bar{f}(-T_1)<0$, $-\frac{T_1+T_2+T_3}{3}<0$ and $\bar{f}(-\frac{T_1+T_2+T_3}{3})>0$, in which 
 case there are two solutions $(g_1,c_1)$ and $(g_2,c_2)$ with $c_1\neq c_2$.   
\end{theorem}
All that follows in this section is the proof of 
Theorem \ref{SO3T}. 
To examine the existence component of Theorem \ref{SO3T}, in accordance with Lemma \ref{BS}, 
first we fix a basis $V_1,V_2,V_3$ in which $T$ is diagonal and look for solutions of 
\eqref{THEQc} or \eqref{THEQ}. Once we have analysed the solvability of \eqref{THEQc} or \eqref{THEQ}, we will then 
investigate uniqueness by examining new bases $X_1,X_2,X_3$ which satisfy the same constraints as $V_1,V_2,V_3$, and use Lemma \ref{UNIQUE}.

\subsection{Existence}
On $SO(3)$, the system of equations \eqref{THEQ} becomes 
\begin{align}\label{THEQSO3}
 (x_2+x_3)x_2x_3=T_1,\qquad
 (x_1+x_3)x_1x_3=T_2,\qquad
 (x_1+x_2)x_1x_2=T_3.
\end{align}
The case where $T$ is positive-definite, so $T_1,T_2,T_3>0$, was treated by Hamilton in~\cite{Hamilton2} and is presented as Lemma \ref{SO3+++}. 
\begin{lemma}\label{SO3+++}
 If $\lambda_1=\lambda_2=\lambda_3=2$ and $T_1,T_2,T_3>0$, \eqref{THEQSO3} has a unique solution $(v_1,v_2,v_3)$ such that $v_1,v_2,v_3>0$.
\end{lemma}
With Lemma \ref{SO3+++} and the equation $v_1v_2v_3c=1$, we can then recover a solution 
$(v_1,v_2,v_3,c)$ of \eqref{THEQc} which is unique up to the scaling of $(v_1,v_2,v_3)$. 

The proof of Lemma \ref{SO3+++} uses a correspondence between solutions of \eqref{THEQSO3}  
and solutions of the equation
\begin{equation}
 f(p)=\frac{p}{p+T_1}+\frac{p}{p+T_2}+\frac{p}{p+T_3}=1.\label{MAIN}
\end{equation}
To see this correspondence, let $p$ be a solution of \eqref{MAIN}
and then define $q$ by 
$q^3=p(p+T_1)(p+T_2)(p+T_3)$. Then 
\begin{equation}\label{SO3p}
 x_1=\frac{q}{p+T_1},\qquad x_2=\frac{q}{p+T_2},\qquad x_3=\frac{q}{p+T_3}
\end{equation}
is a solution of \eqref{THEQSO3}.

Conversely, if $x_1,x_2,x_3$ is a solution of \eqref{THEQSO3}, then let $p=x_1x_2x_3$ and $q=p(x_1+x_2+x_3)$. 
There are then two possibilities: either $x_1x_2x_3(x_1+x_2+x_3)=0$ or 
\begin{align*}
 x_1=\frac{x_1x_2x_3(x_1+x_2+x_3)}{x_2x_3(x_1+x_2+x_3)}=\frac{q}{p+T_1},\qquad
 x_2=\frac{q}{p+T_2},\qquad
 x_3=\frac{q}{p+T_3},
\end{align*} 
and $p$ satisfies \eqref{MAIN}. The condition that $T_1,T_2,T_3,v_1,v_2,v_3$ are all positive is incompatible with the 
equation $x_1x_2x_3(x_1+x_2+x_3)=0$, so this possibility is discarded, and solutions of \eqref{THEQSO3} are found using equation
\eqref{MAIN}.
 
Now we will examine other possibilities for the signature of $T$. From \eqref{THEQc} we can see that the only other possibilities for the signs of $T_1,T_2,T_3$, given our ordering conventions, are 
$T_1>0$ and $T_3\le T_2<0$, or $T_1>0$ and $T_2=T_3=0$. In the case that $T_1>0$ and $T_2=T_3=0$, we 
will not look at \eqref{THEQSO3}, but the original system \eqref{THEQc}.
\begin{lemma}\label{SO3+00}
 If $\lambda_1=\lambda_2=\lambda_3=2$, $T_1>0$ and $T_2=T_3=0$, then the system 
 of equations \eqref{THEQc}
has a solution $(v_1,v_2,v_3,c)$ such that $v_1,v_2,v_3>0$. There are infinitely many solutions, 
but $c=\frac{8}{T_1}$ for all solutions.
\end{lemma}
\begin{proof}
Since $T_2=T_3=0$, \eqref{THEQc} implies that $x_1=0$, or $x_2=x_3=0$. We cannot have 
$x_2=x_3=0$ because then $v_1$ would be $0$. Therefore, $x_1=v_2+v_3-v_1=0$, which implies that 
$x_2=2v_3$ and $x_3=2v_2$. In this case, the first equation of 
\eqref{THEQc} becomes $8=cT_1$ so $c$
must be chosen accordingly, and is the same for all solutions. 

However, the constraint $x_1=0$ only implies that $v_1=v_2+v_3$, so any 
$(v_2+v_3,v_2,v_3,\frac{8}{T_1})$ is a solution of \eqref{THEQc}.
\end{proof}

In the second case that $T_1>0$ and $T_3\le T_2<0$, we know from \eqref{THEQc} that none of $x_1,x_2$ or $x_3$ can be $0$ and $x_1+x_2+x_3=v_1+v_2+v_3>0$. 
Therefore, solutions of 
\eqref{THEQSO3} are in correspondence with solutions of \eqref{MAIN} so we will
look for solutions of \eqref{MAIN} as Hamilton did in Lemma \ref{SO3+++}. 
\begin{lemma} \label{SO3+--}
 If $\lambda_1=\lambda_2=\lambda_3=2$, $T_1>0$ and $T_2,T_3<0$, let 
 $\bar{f}:\mathbb{R}\to\mathbb{R}$ be defined as 
 \begin{equation}
  \bar{f}(p)=2p^3+(T_1+T_2+T_3)p^2-T_1T_2T_3.
 \end{equation}
Then \eqref{THEQSO3} has a solution for 
 $(x_1,x_2,x_3)$ such that $v_1,v_2,v_3>0$ if and only if 
 
 (i) $\bar{f}(-\frac{T_1+T_2+T_3}{3})=\frac{1}{27}(T_1+T_2+T_3)^3-T_1T_2T_3\ge 0$ while $-\frac{T_1+T_2+T_3}{3}<0$, or 
 
 (ii) $\bar{f}(-T_1)=-T_1(T_1-T_2)(T_1-T_3)>0$.
 
 The solution is unique unless 
  $\bar{f}(-T_1)<0$, $-\frac{T_1+T_2+T_3}{3}<0$ and $\bar{f}(-\frac{T_1+T_2+T_3}{3})>0$, in which case there 
  are two solutions.
\end{lemma}
\begin{remark}
It is easily verified that both conditions (i) and (ii) are achieved by some 
$T$. Also, if, for instance, 
$T_1=10$ and $T_2=T_3=-1$, then $\bar{f}$ satisfies the conditions to have two solutions, so non-uniqueness is attainable. 
\end{remark}  
\begin{proof}
For this proof, we will find a solution $p$ of \eqref{MAIN}. However, first we need to find under what circumstances the 
solution will result in $v_1,v_2,v_3$ all being positive. 

We can see that if we have a solution of \eqref{THEQSO3}, then $v_1,v_2,v_3$ are positive if and only if $x_2$ and $x_3$ have the same sign, and 
$x_1$ has a different sign. If this is the case, the numbers $x_2$ and $x_3$ must be positive because $x_2+x_3=2v_1$. 
Once we find a solution of \eqref{MAIN}, the numbers $x_1,x_2,x_3$ solving \eqref{THEQSO3} are given by \eqref{SO3p}
where $q^3=p(p+T_1)(p+T_2)(p+T_3)$. 
Therefore, the solution of \eqref{THEQSO3} has $x_2$ and $x_3$ positive and $x_1$ negative if and only if the corresponding solution $p$ of~\eqref{MAIN} 
satisfies
$-T_1<p<0<-T_2\le -T_3$. 

Solving \eqref{MAIN} on the interval $(-T_1,0)$ is equivalent to solving the equation 
\begin{equation}\label{MAIN'}
\bar{f}(p)=2p^3+(T_1+T_2+T_3)p^2-T_1T_2T_3=0
\end{equation}
on the interval $(-T_1,0)$. Since this is a cubic polynomial, $\bar{f}$ has at most two critical points and will be
monotone on any interval not containing a critical point. Therefore, to evaluate how many
solutions $p$ of \eqref{MAIN'} exist in the interval $(-T_1,0)$, it suffices to check the value of 
$\bar{f}(p)$ at $-T_1$, $0$ and the critical points. After straightforward computation, we conclude that 
a solution exists if and only if $(i)$ or $(ii)$ holds, and that 
the solution is unique unless 
$\bar{f}(-T_1)<0$, $-\frac{T_1+T_2+T_3}{3}<0$ and $\bar{f}(-\frac{T_1+T_2+T_3}{3})>0$. 

\end{proof}
Once we have a $(v_1,v_2,v_3)$ solution of \eqref{THEQSO3} with Lemma \ref{SO3+--}, this triplet alongside 
the positive constant $c$ found with $v_1v_2v_3c=1$ is then a solution of \eqref{THEQc}. This solution of 
\eqref{THEQc} is unique up to the scaling of $v_1,v_2,v_3$. 

We have established that in some instances, there are two solutions of \eqref{THEQSO3}. This means that we have two different solutions 
$(v_1,v_2,v_3,c)$ of \eqref{THEQc}.  
We claim that the two values for $c$ in these two solutions are different. Indeed, if we 
have a solution of \eqref{THEQSO3} then we have a solution of \eqref{THEQc} such that $v_1v_2v_3c=1$. 
If we multiply each of the equations in \eqref{THEQc} by $\frac{x_i}{v_i}$, since $x_1x_2x_3=p$ and $v_1v_2v_3c=1$, we find that 
for each $i$, $\frac{x_i}{v_i}=\frac{2p}{T_i}$. Putting these 
relations back into any one of the three equations of \eqref{THEQc} reveals that $8p^2=cT_1T_2T_3$ so $c$ must be chosen 
accordingly. Since the two different solutions came 
about because there were two different negative values of $p$, we must get two different associated values for $c$. 

\subsection{Uniqueness}
We have now concluded the existence component of 
Theorem \ref{SO3T}. To address the issue of 
uniqueness we will check the hypothesis of Lemma \ref{UNIQUE}. To do this
we need to consider the non-uniqueness of bases which satisfy \eqref{NICE} in which $T$ is diagonal and $T_1\ge T_2\ge T_3$ is satisfied. 
\begin{lemma}\label{SO3Change}
 Let $V_1,V_2,V_3$ be a basis which satisfies \eqref{NICE} in which $T$ is diagonal with $T_1\ge T_2\ge T_3$. If $X_1,X_2,X_3$ is another basis 
 in which the same properties hold, then 
 $X_i=\sum_{j=1}^{3}a_{ji}V_j$ for some $SO(3)$ matrix $(a_{ij})_{i,j=1}^3$  such that 
 $a_{ij}=0$ whenever $T_i\neq T_j$. We also have 
 $T(X_i,X_i)=T(V_i,V_i)=T_i$ for $i=1,2,3$. 
\end{lemma}
\begin{proof}
Since the change preserves the 
Lie bracket relations \eqref{NICE}, the matrix $a=(a_{ij})_{i,j=1}^3$ must be an $SO(3)$ matrix. There are various arguments available to show this. One way is by creating a cross product on the 
Lie algebra which is the same as the Lie bracket. We can then use the orthogonality properties of the cross product to obtain the result. 

It is then straightforward to show that for $T$ to remain diagonal, we must have $a_{ij}=0$ whenever $T_i\neq T_j$, and 
the representation of $T$ is the same in the basis $X_1,X_2,X_3$. 
\end{proof}

We are now in a position to prove the uniqueness part of Theorem \ref{SO3T}. Suppose we have a tensor field $T$ which is diagonal 
in a $V_1,V_2,V_3$ basis with entries 
$T_1,T_2,T_3$. We need to check that the hypotheses of Lemma \ref{UNIQUE} hold, so 
we need to establish 
whether or not the metrics $g$ of our solutions $(g,c)$ 
are diagonal in other bases in which 
$T$ is diagonal and \eqref{NICE} is satisfied. We do this by cases of how many of $T_1$, $T_2$ and $T_3$ are identical. 

\vspace{0.3cm}

\noindent 
$\textbf{Case}$ $\textbf{One}$- $T_1$, $T_2$ and $T_3$ are pairwise distinct. 

In this case, Lemma \ref{SO3Change} implies that the only changes of basis which 
preserve the relations \eqref{NICE} are scalings of the $V_1,V_2,V_3$ basis so 
any diagonal metric in the $V_1,V_2,V_3$ basis is trivially diagonal in the new basis. 
Therefore, Lemma \ref{UNIQUE}, coupled with Remark \ref{+--an}, implies that the uniqueness of solutions 
is completely determined by the uniqueness of solutions of \eqref{THEQSO3}. 

\vspace{0.3cm}

\noindent 
$\textbf{Case}$ $\textbf{Two}$- Exactly two of $T_1$, $T_2$, $T_3$ are the same, say $T_i=T_j$ and let $k$ be the remaining index. 
In this instance, Lemma \ref{SO3Change} tells us that our change of basis is an $SO(3)$ matrix such that $a_{ik}=a_{jk}=a_{ki}=a_{kj}=0$, so 
the representation of $T$ is the same.

Now since $T_i=T_j$, any solution $(v_1,v_2,v_3)$ of 
\eqref{THEQSO3} must satisfy $v_i=v_j$ or $x_k=v_i+v_j-v_k=0$. 
In the case that $x_k=v_i+v_j-v_k=0$ we find that two of $T_1,T_2,T_3$ are $0$. By our ordering convention 
this means that we are in the situation where $x_1=0$, $T_1>0$ and $T_2=T_3=0$. The solutions of $Ric(g)=cT$ are non-unique because 
the solutions of \eqref{THEQc} are non-unique. However, we note that $T(X_1,X_1)=T_1$ for any change of basis 
which preserves \eqref{NICE} in which $T$ remains diagonal. Since the solution of \eqref{THEQc} satisfies 
$c=\frac{8}{T_1}$, the value of $c$ is the same for all solutions. 

In the case that $v_i=v_j$ and none of $T_1,T_2,T_3$ is $0$, computation reveals that the change of basis leaves the metric diagonal. Therefore, Lemma \ref{UNIQUE} tells us that the pair $(g,c)$ solving $Ric(g)=cT$ is unique, or, by Remark 
\ref{+--an}, there are two solutions if 
$T$ has signature $(+,-,-)$ and $T$ satisfies the conditions of Lemma \ref{SO3+--} for two solutions.

\vspace{0.3cm}

\noindent 
$\textbf{Case}$ $\textbf{Three}$- $T_1=T_2=T_3$. In this case, all three of $T_1,T_2,T_3$ must be positive to have a solution of \eqref{THEQSO3}. In this case, similarly to the case 
that two of $T_1,T_2,T_3$ are the same, we find that 
$\frac{x_1}{v_1}=\frac{x_2}{v_2}=\frac{x_3}{v_3}$ and since none of $x_1,x_2,x_3$ can be $0$, our solution of \eqref{THEQSO3} 
must also satisfy $v_1=v_2=v_3$. Therefore any change of basis that leaves $T$ diagonal will also leave the metric solution of $Ric(g)=cT$ diagonal and 
Lemma \ref{UNIQUE} can be applied. 

This treats the issue of uniqueness and concludes our analysis on $SO(3)$. 
\section{The Special Linear Group $SL(2)$}
For $SL(2)$, we have $\lambda_1=\lambda_2=2$ and $\lambda_3=-2$. The main result of this section 
is presented in the following theorem. 
 \begin{theorem}\label{SL2T}
Let $T$ be a left-invariant $(0,2)$ tensor on $SL(2)$. There 
exists a left-invariant Riemannian metric $g$ and a positive constant $c>0$ such that 
$Ric(g)=cT$ if and only if $T$ is diagonalisable with $T=diag(T_1,T_2,T_3)$ in a basis satisfying \eqref{NICE} with 
$\lambda_1=\lambda_2=2$, $\lambda_3=-2$ and one of the following conditions is satisfied:

(i) $T_1>0$, $T_2,T_3<0$, $-T_1<T_3$

(ii) $T_2>0$, $T_1,T_3<0$, $-T_2<T_3$

(iii) $T_1,T_2<0$, $T_3>0$, $\max\{-T_1,-T_2\}<T_3$

(iv) $T_1,T_2<0$, $T_3>0$, $\min\{-T_1,-T_2\}>T_3$

(v) $T_1=T_2=-T_3<0$

(vi) $T_1<0$, $T_2=T_3=0$

(vii) $T_2<0$, $T_1=T_3=0$.

In the cases (i)-(iv), the solution $(g,c)$ of $Ric(g)=cT$ is unique. In (v)-(vii), there are infinitely many 
solutions $(g,c)$, but 
$c$ is the same for all solutions. 
\end{theorem}
What follows in this section is the proof of Theorem \ref{SL2T}. 
To prove the existence component of this theorem, 
we will use Lemma \ref{BS}, so we fix a basis $V_1,V_2,V_3$ in which $T$ is diagonal and \eqref{NICE} is satisfied. 
Once such a basis is found, we look for solutions of \eqref{THEQc} or 
\eqref{THEQ}.
To examine the uniqueness component of Theorem \ref{SL2T}, 
we will then consider other bases $X_1,X_2,X_3$ which satisfy \eqref{NICE} in which $T$ is also diagonal in order to examine the uniqueness 
component of Theorem~\ref{SL2T} in accordance with Lemma \ref{UNIQUE}.

\subsection{Existence}
Similarly to what we found for $SO(3)$, solutions of \eqref{THEQ} are in correspondence with solutions of 
\begin{equation}
 f(p)=\frac{p}{p+T_1}+\frac{p}{p+T_2}+\frac{p}{p-T_3}=1\label{MAIN2}.
\end{equation}
The correspondence is described with  
\begin{equation}\label{x_1x_2x_3}
 x_1=\frac{q}{p+T_1},\qquad
 x_2=\frac{q}{p+T_2},\qquad
 x_3=\frac{q}{p-T_3}
\end{equation}
where $q$ is given by
$q^3=p(p+T_1)(p+T_2)(p-T_3)$. Similarly to $SO(3)$, the correspondence is one-to-one unless there is a solution of \eqref{THEQ} satisfying $x_1x_2x_3(x_1+x_2+x_3)=0$. 

Now to prove the existence component of Theorem \ref{SL2T}, we need to solve equation~\eqref{THEQc} or equation \eqref{THEQ}. 
We will treat these by the signs of $T_1,T_2,T_3$. Also note that when solving \eqref{THEQ}, 
we will use the equivalence between 
\eqref{THEQ} and \eqref{MAIN2} that we just discussed. 

\vspace{0.3cm}
\noindent 
$\textbf{Case}$ $\textbf{One}$- none of $T_1,T_2,T_3$ is $0$. 

We can see from \eqref{THEQc} that we cannot have any of $x_1,x_2,x_3$ being $0$, and 
in order for $v_1,v_2,v_3$ to be positive, two of $T_1,T_2,T_3$ are negative and the third is positive. 
Furthermore, if $x_1+x_2+x_3=0$, then $v_3=v_1+v_2$ and \eqref{THEQc} implies that 
$T_1=T_2=-T_3$. 
Therefore, in all cases except where $T_1=T_2=-T_3<0$,  $x_1x_2x_3(x_1+x_2+x_3)\neq 0$ so solving \eqref{MAIN2} is equivalent to solving \eqref{THEQ}.

\begin{lemma}\label{SL2+--}
 If $\lambda_1=\lambda_2=2$, $\lambda_3=-2$, $T_1>0$ and $T_2,T_3<0$, then \eqref{THEQ} has a solution such that $v_1,v_2,v_3>0$ if and only if $-T_1<T_3$. If 
 $T_2>0$ and $T_1,T_3<0$, then~\eqref{THEQ} has a solution 
 with $v_1,v_2,v_3>0$ if and only if $-T_2<T_3$. The solution is unique in both cases. 
\end{lemma}
\begin{proof}
We construct a solution of \eqref{THEQ} by finding 
a solution of \eqref{MAIN2} and checking that $v_1,v_2,v_3$ associated to this solution are all positive.

First let us suppose that $T_1>0$ and $T_2,T_3<0$. In this case we note that $f(p)>1$ 
if $p<\min\{-T_1,T_3,-T_2\}$, $f(p)<1$ if $\max\{-T_1,T_3\}<p<0$, and 
$f(p)>1$ if $p>-T_2$. Thus the only possible solutions of 
\eqref{MAIN2} are $p$ between $T_3$ and $-T_1$, 
or between $0$ and $-T_2$. From \eqref{x_1x_2x_3} and \eqref{THEQc}, we see that the solution needs to be between 
$T_3$ and $-T_1$ to have $v_1,v_2,v_3$ all positive, and we also require $-T_1<T_3$. 
If indeed we do have $-T_1<T_3$, then as 
$p$ approaches $-T_1$ from above, $f(p)$ approaches $-\infty$, and as $p$ approaches $T_3$ from below, $f(p)$ approaches $+\infty$, 
so we have a solution of \eqref{MAIN2}. 
 
We claim that this solution is unique. To see this, similarly to what we did 
for $SO(3)$, 
convert $f(p)=1$ into the cubic equation
\begin{equation}
 \bar{f}(p)=2p^3+(T_1+T_2-T_3)p^2+T_1T_2T_3=0.
\end{equation}
It is straightforward to show that the solution $p\in(-T_1,T_3)$ is unique by analysing the behaviour of the critical points of 
$\bar{f}$.

If $T_2>0$ and $T_1,T_3<0$, similar reasoning reveals that for positive $v_1,v_2,v_3$, we require $-T_2<T_3$, and that there is indeed a unique solution of \eqref{MAIN2} between $-T_2$ and $T_3$. 
\end{proof}

Let us now consider the case where $T_1,T_2<0$ and $T_3>0$.
\begin{lemma}\label{SL2--+}
  Suppose $\lambda_1=\lambda_2=2$, $\lambda_3=-2$, $T_1,T_2<0$, $T_3>0$ and the two equations 
  $T_3=-T_1$ and $T_3=-T_2$ do not hold simultaneously. Then
 \eqref{THEQ} has a solution satisfying $v_1,v_2,v_3>0$ if and only if $\max\{-T_1,-T_2\}<T_3$ or $T_3<\min\{-T_1,-T_2\}$. This solution is unique.
\end{lemma}
\begin{proof}
From \eqref{THEQc} and \eqref{x_1x_2x_3} we see that our solution $p$ of 
\eqref{MAIN2} corresponds to $v_1,v_2,v_3$ all being positive if and only if $\max\{-T_1,-T_2\}<p<T_3$ or $T_3<p<\min\{-T_1,-T_2\}$.
If $\max\{-T_1,-T_2\}<T_3$ then we have a unique solution because $f(p)$ is monotonically decreasing on the interval $(\max\{-T_1,-T_2\},T_3)$, with 
$f(p)$ approaching $\infty$ and $-\infty$ as $p$ approaches $\max\{-T_1,-T_2\}$ and $T_3$, respectively. 
Similarly, if $T_3<\min\{-T_1,-T_2\}$ a solution exists and is unique.
\end{proof}
\begin{lemma}\label{SL2--+=}
  Suppose $\lambda_1=\lambda_2=2$, $\lambda_3=-2$, $T_1,T_2<0$, $T_3>0$ and $T_1=-T_3$ or $T_2=-T_3$. Then the system of equations \eqref{THEQc} has a solution $(v_1,v_2,v_3,c)$ where 
  $v_1,v_2,v_3>0$ if and only if $T_1=T_2=-T_3$.  
  There are infinitely many solutions $(v_1,v_2,v_3,c)$, but $c$ is the same for all solutions. 
\end{lemma}
\begin{proof}
If exactly one of the two equations $T_3=-T_1$ or $T_3=-T_2$ holds then Lemma 
  \ref{SL2--+} implies that there is no solution of \eqref{THEQ}, 
  hence there is no solution of \eqref{THEQc}. We can therefore assume that $T_1=T_2=-T_3<0$. In this case, \eqref{THEQc} implies that 
  $v_1+v_2=v_3$, and $c=\frac{8}{T_3}$, which are sufficient constraints for a solution. 
\end{proof}

\vspace{0.3cm}
\noindent
$\textbf{Case}$ $\textbf{Two}$- at least one of $T_1,T_2,T_3$ is $0$. 

In this case, we see from \eqref{THEQc} that in fact two of $T_1,T_2,T_3$ must be $0$ and the 
third, which cannot be $T_3$, is negative. 
\begin{lemma}\label{SL20}
Suppose $\lambda_1=\lambda_2=2$, $\lambda_3=-2$ and at least one of $T_1,T_2,T_3$ is $0$. Then the system of equations \eqref{THEQc}
has a solution $(v_1,v_2,v_3,c)$ such that 
$v_1,v_2,v_3>0$ if and only if $T_1<0$ and $T_2=T_3=0$, or $T_2<0$ and $T_1=T_3=0$. There are infinitely many solutions $(v_1,v_2,v_3,c)$, 
but $c$ is the same for all solutions and is given by $c=-\frac{T_i}{8}$ where $i$ is the unique index such that $T_i<0$. 
\end{lemma}
\begin{proof}
The proof is straightforward and is very similar to the proof of Lemma \ref{SO3+00}. 
\end{proof}

\subsection{Uniqueness}
We have now shown exactly when equations \eqref{THEQc} and \eqref{THEQ} have solutions for positive $v_1,v_2,v_3$. 
By Lemma \ref{BS}, this settles the existence component of Theorem \ref{SL2T}. 
Now to examine uniqueness, we must consider all bases $X_1,X_2,X_3$ which satisfy 
\eqref{NICE} for $\lambda_1=\lambda_2=-\lambda_3=2$ 
in which $T$ remains diagonal. To use Lemma~\ref{UNIQUE} we need to show that the $g$ solution of $Ric(g)=cT$ which is diagonal in the $V_1,V_2,V_3$ basis is also diagonal in the $X_1,X_2,X_3$ basis.
First we will consider when the change of basis satisfies \eqref{NICE}. 
\begin{lemma}\label{SL2C}
Let $V_1,V_2,V_3$ be a basis of the 
Lie algebra of $SL(2)$ which satisfies~\eqref{NICE} for $\lambda_1=\lambda_2=2$ and $\lambda_3=-2$.
If $X_1,X_2,X_3$ is another basis 
that satisfies~\eqref{NICE} then 
the change of basis matrix from $V_1,V_2,V_3$ to $X_1,X_2,X_3$ has the form
\begin{equation}\label{SL2CE}
 bda=
 \begin{pmatrix}
  \cos(\theta)&-\sin(\theta)&0\\
  \sin(\theta)&\cos(\theta)&0\\
  0&0&1
 \end{pmatrix}
 \begin{pmatrix}
  1&0&0\\
  0&\cosh(\phi)&-\sinh(\phi)\\
  0&-\sinh(\phi)&\cosh(\phi)
 \end{pmatrix}
 \begin{pmatrix}
  0&a_{12}&a_{13}\\
 \pm 1& 0&0\\
 0& \mp a_{13}&\mp a_{12}
\end{pmatrix}
\end{equation}
for some real parameters $\theta,\phi,a_{12}$ and $a_{13}$ such that $a_{12}^2-a_{13}^2=1$. 
\end{lemma}
\begin{proof}
Suppose that we have a basis $V_1,V_2,V_3$ that satisfies \eqref{NICE}. Suppose also that the basis $X_i=\sum_{j=1}^{3}A_{ji}V_j$ satisfies \eqref{NICE}. 

Changes of basis given by the matrices $b^{-1}$ and $d^{-1}$ preserve the Lie bracket 
relations \eqref{NICE} for any $\theta$ and $\phi$. Therefore, it is clear that 
the change of basis corresponding to $A$ preserves \eqref{NICE} if and only if the change
$a=d^{-1}b^{-1}A$ also preserves \eqref{NICE}. By choosing $\theta$ and $\phi$ appropriately, we can assume that 
$a_{11}=0$. By refining our choice of $\phi$, we can also assume other constraints on $a$ so that finding when $a$ preserves \eqref{NICE} becomes straightforward. 
\end{proof}

Suppose that $T$ is diagonal in a basis $V_1,V_2,V_3$. We know from Lemma \ref{SL2C} 
that any change of basis matrix which preserves
the Lie bracket relations has the form 
$bda$, where $b$, $d$ and $a$ are given by \eqref{SL2CE}.

To apply Lemma \ref{UNIQUE}, we need to find out when $T$ remains diagonal under such a change, 
and how our metric solutions behave under such a change.
The following lemma provides initial constraints on $T$ and our change 
of basis. 
\begin{lemma}\label{SL2S}
 Suppose a left-invariant $(0,2)$ tensor field $T$ is diagonal in a basis $V_1,V_2,V_3$. Let 
 $Y_1,Y_2,Y_3$ be the basis related to the basis $V_1,V_2,V_3$ by the matrix $bd$ and let 
 $X_1,X_2,X_3$ be the basis related to $Y_1,Y_2,Y_3$ by the matrix $a$. If $T$ is diagonal in the basis $X_1,X_2,X_3$ 
 then $T$ is diagonal in the basis $Y_1,Y_2,Y_3$ and at least one of the following conditions holds: 
 
 (i) $T_1=T_2$, where $T_i$ is given by \eqref{TV}
 
 (ii) $\cos(\theta)=0$
 
 (iii) $\sin(\theta)=0$.
\end{lemma}
\begin{proof}
Since the representation of $T$ becomes diagonal after the $a$ change, we find that 
\begin{align*}
T(X_1,X_2)&=T(\pm Y_2,a_{12}Y_1\mp a_{13}Y_3)=\pm a_{12}T(Y_1,Y_2)\mp a_{13}T(Y_2,Y_3)=0,\\
T(X_1,X_3)&=T(\pm Y_2,a_{13}Y_1\mp a_{12}Y_3)=\pm a_{13}T(Y_1,Y_2)\mp a_{12}T(Y_2,Y_3)=0,
\end{align*}
so 
\begin{align*}
 a_{12}T(Y_1,Y_2)= a_{13}T(Y_2,Y_3),\qquad
 a_{13}T(Y_1,Y_2)= a_{12}T(Y_2,Y_3).
\end{align*}
If either one of $T(Y_2,Y_3)$ or $T(Y_1,Y_2)$ were not $0$, we could multiply the first equation by $a_{13}$ or 
the second by $a_{12}$ respectively, and substitute one into the other. We would conclude that $a_{12}^2=a_{13}^2$ which conflicts with the 
constraint $1=-a_{13}^2+a_{12}^2$, so both $T(Y_1,Y_2)$ and $T(Y_2,Y_3)$ are $0$. It therefore remains to check that $T(Y_1,Y_3)=0$.

Now the $Y_1,Y_2,Y_3$ basis is related to the $V_1,V_2,V_3$ by the change of basis matrix
\begin{align*}
 bd=
 \begin{pmatrix}
  \cos(\theta)&-\sin(\theta)\cosh(\phi)&\sin(\theta)\sinh(\phi)\\
  \sin(\theta)&\cos(\theta)\cosh(\phi)&-\cos(\theta)\sinh(\phi)\\
  0&-\sinh(\phi)&\cosh(\phi)
 \end{pmatrix}
\end{align*}
so we find that 
\begin{align*}
 T(Y_1,Y_2)&=-\cos(\theta)\sin(\theta)\cosh(\phi)T_1+\sin(\theta)\cos(\theta)\cosh(\phi)T_2=0.
\end{align*}
Since $\cosh(\phi)\neq 0$, this equation implies that $T_1=T_2$, $\sin(\theta)=0$ or $\cos(\theta)=0$. In any case,
we find that $T(Y_1,Y_3)=\cos(\theta)\sin(\theta)\sinh(\phi)(T_1-T_2)=0$ so $T$ must in fact be diagonal in the 
$Y_1,Y_2,Y_3$ basis. 
\end{proof}

We use Lemma \ref{SL2S} to treat uniqueness by cases. 

\vspace{0.3cm}
\noindent
$\textbf{Case}$ $\textbf{One}$- $T_1\neq T_2$. 

This case is treated by Lemmas \ref{SL2Y} and \ref{SL2N}. 
For Lemma \ref{SL2Y} we suppose that none of
$T_1,T_2,T_3$ is $0$ and in Lemma \ref{SL2N} we suppose that at least one of $T_1,T_2,T_3$ is $0$. 
The proof of both lemmas is by use of Lemma \ref{SL2S}, direct evaluation and the solving of simple simultaneous equations. 
\begin{lemma}\label{SL2Y}
Suppose there is some left-invariant $T$ such that $Ric(g)=cT$ for some left-invariant $g$ and some constant $c>0$. Let $V_1,V_2,V_3$ 
be a basis satisfying~\eqref{NICE} in which $g$ and $T$ are diagonal. 
Suppose that $T_1\neq T_2$ and none of $T_1,T_2,T_3$ is $0$, where $T_i$ is given by \eqref{TV}. If $T$ is diagonal in some
other basis $X_1,X_2,X_3$ which preserves the Lie bracket relations \eqref{NICE},
then $X_i=\pm V_i$ after possibly interchanging $V_1$ and $V_2$.
\end{lemma}
\begin{lemma}\label{SL2N}
Suppose there is some left-invariant $T$ such that $Ric(g)=cT$ for some left-invariant $g$ and some constant $c>0$. Let $V_1,V_2,V_3$ 
be a basis satisfying \eqref{NICE} in which $g$ and $T$ are diagonal. 
Suppose that $T_1\neq T_2$ and at least one of $T_1,T_2,T_3$ is $0$, where $T_i$ is given by \eqref{TV}. If $T$ is diagonal in some
other basis $X_1,X_2,X_3$ which preserves the Lie bracket relations \eqref{NICE},
then the representation of $T$ is the same, apart from possibly interchanging $T_1$ and $T_2$.
\end{lemma}
With Lemma \ref{SL2Y}, it is straightforward to verify that if $Ric(g)=cT$ for some $T$ such that $T_1\neq T_2$ and none of $T_1,T_2,T_3$ is $0$, then 
the hypothesis of Lemma \ref{UNIQUE} holds because any allowable change of basis is trivial. Therefore the pair 
$(g,c)$ solving $Ric(g)=cT$ is unique. If one of $T_1,T_2,T_3$ is $0$, there are infinitely many solutions, but Lemma \ref{SL2N} shows that the $c$ value is the same for all solutions.

\vspace{0.3cm}
\noindent
$\textbf{Case}$ $\textbf{Two}$- $T_1= T_2$. 

 In this case, the change of basis corresponding to the $b$ matrix leaves the representation of $T$ the same as it 
was before. Then Lemma \ref{SL2S} implies that 
\begin{align*}
T(Y_2,Y_3)=-\cosh(\phi)\sinh(\phi)T_2-\cosh(\phi)\sinh(\phi)T_3=0,
\end{align*}
 so we must have $T_2+T_3=T_1+T_3=0$ or $\sinh(\phi)=0$. 

If $T_2+T_3=T_1+T_3=0$ then we have established that 
the solutions of \eqref{THEQc} are those which satisfy $v_3=v_1+v_2$ and $c=\frac{8}{T_3}$, so the solutions are non-unique and 
Lemma \ref{UNIQUE} implies that the solutions $(g,c)$ of $Ric(g)=cT$ are non-unique. 
However, direct computation reveals that in any allowable change of basis, the value of $T_3$ remains unchanged, so 
$c$ is the same for all solutions. 

Our final case is that $T_1=T_2$, but $T_2+T_3\neq 0$, so $\sinh(\phi)=0$.
\begin{lemma}
Suppose $T$ is diagonal 
in a basis $V_1,V_2,V_3$ such that $T_1=T_2$ and $T_2+T_3\neq 0$ where $T_i$ is given in \eqref{TV}.
If $T$ is diagonal in any other basis $X_1,X_2,X_3$ which preserves the Lie bracket relations, 
then any metric which is diagonal in the $V_1,V_2,V_3$ basis with components solving \eqref{THEQ} is also diagonal in the basis $X_1,X_2,X_3$. 
\end{lemma}
\begin{proof}
Lemma \ref{SL2S} implies that $\sinh(\phi)=0$ so $\phi=0$ and the $d$ matrix is the identity. 
In this case, computation reveals that 
$T$ is diagonal after the $bd$ change and has the same components. Now we know that 
$T$ is diagonal in the 
$X_1,X_2,X_3$ basis, but since $X_1,X_2,X_3$ is related to $Y_1,Y_2,Y_3$ by the 
$a$ matrix and the representation of $T$ in the $Y_1,Y_2,Y_3$ is still $diag(T_1,T_2,T_3)$, we find that $T(X_2,X_3)=a_{12}a_{13}(T_2+T_3)=0$.
Therefore, one of $a_{12}$ or $a_{13}$ is $0$.

Since $T_1=T_2$, the equations of \eqref{THEQc} reveal to us that $\frac{x_1}{v_1}=\frac{x_2}{v_2}$ or $x_3=v_1+v_2+v_3=0$. Since $x_3$ must be positive, we conclude that $v_1=v_2$ or $v_3=v_1+v_2$. We can exclude the second case $v_3=v_1+v_2$ since $T_2+T_3\neq 0$, so we can now assume that $v_1=v_2$.  
Due to this constraint, the 
diagonal metric remains diagonal after the matrix $b$ change. It trivially remains diagonal after the $d$ change and the $a$ change since $\phi=0$ and 
one 
of $a_{12}$ or $a_{13}$ is $0$. 
\end{proof}
If $T_1=T_2$ and $T_2+T_3\neq 0$, in order to have a solution of 
\eqref{THEQ}, the numbers $T_1,T_2,T_3$ must satisfy the conditions of Lemma 
\ref{SL2--+} so the solution of 
\eqref{THEQ} is unique. The previous proof establishes that $\phi=0$
and one of $a_{12}$ or $a_{13}$ equals $0$. Due to the constraint $1=a_{12}^2-a_{13}^2$, 
we conclude that $a_{13}=0$. Therefore, our change of basis is simply an interchanging of the vectors $V_1,V_2$, so 
the representation of $T$ in the new basis is exactly the same as the old. Therefore 
the new components of $T$ still satisfy the hypothesis of Lemma \ref{SL2--+} so the
solution of \eqref{THEQ} is again unique and Lemma \ref{UNIQUE} implies that the pair $(g,c)$
solving $Ric(g)=cT$ is unique.  

This concludes our analysis on $SL(2)$. 
\section{The Remaining Four Three-Dimensional Unimodular Lie Groups}
The remaining four three-dimensional unimodular Lie groups are the Euclidean group $E(2)$, the Minkowski group $E(1,1)$, the 
Heisenberg group $H_3$, and the group $\mathbb{R}^3$. 

The following theorems are the main results for the first three of these Lie groups.  
\begin{theorem}\label{E2T}
Let $T$ be a left-invariant $(0,2)$ tensor field on $E(2)$. There 
exists a left-invariant Riemannian metric $g$ and a positive constant $c>0$ such that 
$Ric(g)=cT$ if and only if $T$ is diagonalisable with $T=diag(T_1,T_2,T_3)$ in a basis satisfying \eqref{NICE} with 
$\lambda_1=\lambda_2=2$, $\lambda_3=0$ and one of the following conditions is satisfied:

(i) $T_1=T_2=T_3=0$

(ii) $T_1+T_2>0$, $T_3<0$, $T_1T_2<0$.

In case (i), there are infinitely many solutions $(g,c)$ of $Ric(g)=cT$. In case (ii), the solution $(g,c)$ is unique. 
\end{theorem}
\begin{theorem}\label{E11T}
Let $T$ be a left-invariant $(0,2)$ tensor field on $E(1,1)$. There 
exists a left-invariant Riemannian metric $g$ and a positive constant $c>0$ such that 
$Ric(g)=cT$ if and only if $T$ is diagonalisable with $T=diag(T_1,T_2,T_3)$ in a basis satisfying \eqref{NICE} with 
$\lambda_1=2$, $\lambda_2=-2$, $\lambda_3=0$ and one of the following conditions is satisfied:

(i) $T_1=T_2=0$, $T_3<0$

(ii) $T_1+T_2>0$, $T_3<0$, $T_1T_2<0$.

In case (i), there are infinitely many pairs $(g,c)$ solving $Ric(g)=cT$, but the $c$ value of all solutions is the same. 
In case (ii), the pair $(g,c)$ is unique.
\end{theorem}
\begin{theorem}\label{HT}
Let $T$ be a left-invariant $(0,2)$ tensor on $H_3$. There 
exists a left-invariant Riemannian metric $g$ and a positive constant $c>0$ such that 
$Ric(g)=cT$ if and only if $T$ is diagonalisable with $T=diag(T_1,T_2,T_3)$ in a basis satisfying \eqref{NICE} with 
$\lambda_1=2$, $\lambda_2=\lambda_3=0$ and $T_1>0$, $T_2<0$ and $T_3<0$. 
The pair $(g,c)$ is unique.  
\end{theorem}

As we have done previously, to prove the existence components of these theorems we fix a basis $V_1,V_2,V_3$ satisfying \eqref{NICE} and solve 
\eqref{THEQc} with $T_i$ given by \eqref{TV}. Since $\lambda_3=0$ for each Lie group, this task is simple and the results are presented in the following 
lemmas. 
\begin{lemma}\label{E2E}
Suppose $\lambda_1=\lambda_2=2$ and $\lambda_3=0$. Then \eqref{THEQc} has a solution satisfying $v_1,v_2,v_3>0$ if 
and only if:

(i) $T_1=T_2=T_3=0$ or

(ii) $T_3<0$, $T_1+T_2>0$ and $T_1T_2<0$. 

The solution of \eqref{THEQc} is non-unique in (i) and is unique in (ii). 
\end{lemma}
\begin{lemma}\label{E11S}
Suppose $\lambda_1=2$, $\lambda_2=-2$ and $\lambda_3=0$. Then \eqref{THEQc} has a solution satisfying $v_1,v_2,v_3>0$ if 
and only if:

(i) $T_1=T_2=0$ and $T_3<0$ or 

(ii) $T_3<0$, $T_1+T_2>0$ and $T_1T_2<0$. 

The solution of \eqref{THEQc} is non-unique in (i) but $c=-\frac{8}{T_3}$ for all solutions. 
The solution is unique in (ii). 
\end{lemma}
\begin{lemma}\label{HS}
Suppose $\lambda_1=2$, and $\lambda_2=\lambda_3=0$. Then \eqref{THEQc} has a solution satisfying $v_1,v_2,v_3>0$ if 
and only if $T_1>0$, $T_2<0$ and $T_3<0$. 
The solution of \eqref{THEQc} is unique. 
\end{lemma}
We will now examine the uniqueness component of 
Theorems \ref{E2T}, \ref{E11T} and \ref{HT}. To do this, we will use Lemma \ref{UNIQUE},
which means we need to find all other bases in which $T$ is diagonal and \eqref{NICE} is satisfied, 
then examine whether our $g$ solution of $Ric(g)=cT$ remains diagonal 
after these changes. The following lemma finds the allowable changes of basis. The proof is similar in style to the proof of 
Lemma~\ref{SL2C}, but is simplified by the fact that in all cases, we have $\lambda_3=0$. 
\begin{lemma}\label{E2CB}
Let $\lambda_1=2$, $\lambda_3=0$ and let $X_i=\sum_{j=1}^{3}a_{ji}V_j$ be a change of basis from a basis 
$V_1,V_2,V_3$ which satisfies 
\eqref{NICE}. If this new basis satisfies the relations \eqref{NICE} then  the following constraints must be satisfied depending on $\lambda_2$.

(i) $\lambda_2=2$: $a_{31}=a_{32}=0$, $a_{33}=\pm 1$, 
$a_{21}=\mp a_{12}$ and $a_{11}=\pm a_{22}$.

(ii) $\lambda_2=-2$: $a_{31}=a_{32}=0$, $a_{33}=\pm 1$, 
$a_{21}=\pm a_{12}$ and $a_{11}=\pm a_{22}$.

(iii) $\lambda_2=0$:   $a_{21}=a_{31}=0$ and $a_{11}=a_{22}a_{33}-a_{23}a_{32}$.
\end{lemma} 
We now know when the Lie bracket relations of \eqref{NICE} are preserved for each Lie group. It follows from elementary 
computation that the hypothesis of Lemma \ref{UNIQUE} holds for any of these Lie groups if none of $T_1,T_2,T_3$ is $0$. 
Similarly, it follows that if $T$ satisfies $(i)$ of Lemma \ref{E11S}, then there are infinitely many solutions of $Ric(g)=cT$, but 
$c$ is the same for all solutions because $T_3$ is unchanged by changes of basis preserving \eqref{NICE} in which $T$ remains diagonal. 

For the last of the four Lie groups, $\mathbb{R}^3$ we have $\lambda_1=\lambda_2=\lambda_3=0$. The Ricci curvature of any left-invariant metric on $\mathbb{R}^3$ is $0$. Therefore, any pair $(g,c)$ solves $Ric(g)=cT$ if $T=0$, and there is 
no solution if $T\neq 0$. 

This concludes our analysis of the solvability of the prescribed 
Ricci curvature problem for left-invariant metrics on the six three-dimensional unimodular Lie groups. 
\section*{Acknowledgements}
I would like to thank Artem Pulemotov, J\o rgen Rasmussen and Joseph Grotowski for the helpful 
discussions. This research was partially supported by the Australian Research Council through Artem Pulemotov's Discovery Early Career Researcher Award DE150101548.

\end{document}